\documentclass[12pt,a4paper]{article}

\usepackage[mathscr]{eucal}
\usepackage{amssymb}
\usepackage{latexsym}
\usepackage{amsthm}
\usepackage{amsmath}
\usepackage[dvips]{graphicx}
\usepackage{psfrag}
\newtheorem{theorem}{Theorem}[section]

\newtheorem{lemma}{Lemma}[section]
\newtheorem{corollary}{Corollary}[section]
\theoremstyle{definition}
\newtheorem{definition}{Definition}[section]

\newtheorem{example}{Example}[section]

\makeatletter
\@addtoreset{equation}{section}

\makeatother

\title{New upper bound for the cardinalities of $s$-distance sets on the unit sphere}
\author{Hiroshi Nozaki}
\date{November, 2008}
\begin{document}
\maketitle
\begin{abstract}
We have the Fisher type inequality and the linear programming bound as upper bounds for the cardinalities of $s$-distance sets on $S^{d-1}$. In this paper, we give a new upper bound for the cardinalities of $s$-distance sets on $S^{d-1}$ for any $s$. This upper bound improves the Fisher typer inequality and is useful for $s$-distance sets which are not applicable to the linear programming bound.     
\end{abstract}
\section{Introduction}
Let $X$ be a finite set on $S^{d-1}$. We define 
$$
A(X):=\{(x, y) \mid x,y \in X, x\ne y\},
$$
where $(\ast,\ast)$ is the standard inner product. $X$ is called an $s$-distance set, if the number of Euclidean distances between any distinct vectors in $X$ is $s$, that is, $|A(X)| = s$. The following upper bound is well known. 
\begin{theorem}[Fisher type inequality \cite{DGS}] 
\begin{enumerate}
\item Let $X$ be an $s$-distance set on $S^{d-1}$. Then, $|X| \leq \binom{d+s-1}{s}+\binom{d+s-2}{s-1}$. 
\item Let $X$ be an antipodal $s$-distance set on $S^{d-1}$. Then, $|X| \leq 2 \binom{d+s-2}{s-1}$. 
\end{enumerate}
\end{theorem}
$X$ is said to be antipodal, if $-X=\{-x \mid x \in X \} \subset X$. We prepare the Gegenbauer polynomial. 
\begin{definition}
Gegenbauer polynomials are a set of
orthogonal polynomials $\{G_l^{(d)}(t) \mid l=1,2,\ldots \}$ of one variable $t$.
For each $l$, $G_l^{(d)}(t)$ is a 
polynomial of degree $l$ and defined in the following manner.
\begin{enumerate}
\item $G_0^{(d)}(t) \equiv 1$, $G_1^{(d)}(t)=d t$.
\item $t G_l^{(d)}(t)=\lambda_{l+1}G_{l+1}^{(d)}(t)+(1-\lambda_{l-1})G_{l-1}^{(d)}(t)$ for $l \geq 1$, where $\lambda_l=\frac{l}{d+2l-2}$.
\end{enumerate}
\end{definition}
The following upper bound for the cardinalities of $s$-distance sets is well known. 
\begin{theorem}[Linear programming bound \cite{DGS}]
Let $X$ be an $s$-distance set on $S^{d-1}$. We define the polynomial $F_X(t):=\prod_{\alpha \in A(X)}(t-\alpha)$ for $X$. 
We have the Gegenbauer expansion 
$$
F_X(t)=\prod_{\alpha \in A(X)}(t-\alpha)= \sum_{k=0}^s f_k G_k^{(d)}(t),
$$
where $f_k$ are real numbers. If $f_0>0$ and $f_i\geq 0$ for all $1 \leq i \leq s$, then
$$
|X| \leq \frac{F_X(1)}{f_0}.
$$
\end{theorem}
This upper bound is very useful when $A(X)$ is given. However, if there exists $f_i$ which is a negative number, then we have no useful upper bound for the cardinalities of $s$-distance sets. In this paper, we give useful upper bounds for that case. 

Let ${\rm Harm}_l(\mathbb{R}^d)$ denote the linear space of all harmonic homogeneous polynomials of degree $l$, in $d$ variables. Let $h_l$ denote the dimension of ${\rm Harm}_l(\mathbb{R}^d)$. The following are the main theorems in this paper. 
\begin{theorem} \label{main}
Let $X$ be an $s$-distance set on $S^{d-1}$.  
We define the polynomial $F_X(t)$ of degree $s$:  
$$F_X(t):=\prod_{\alpha \in A(X)} (t-\alpha)= \sum_{i=0}^s f_i G_i^{(d)}(t),$$
where $f_i$ are real number. Then, 
\begin{eqnarray}
|X| \leq \sum_{i \text{ with } f_i>0} h_i, 
\end{eqnarray} 
where the summation runs through $0\leq i \leq s$ satisfying $f_i>0$.  
\end{theorem}

\begin{theorem}[Antipodal case] \label{main anti}
 Let $X$ be an antipodal $s$-distance set on $S^{d-1}$. There exists $Y$ such that $X=Y \cup (-Y)$ and $|X|=2|Y|$. 
 We define the polynomial $F_Y(t)$ of degree $s-1$:  
$$F_Y(t):=\prod_{\alpha \in A(Y)} (t-\alpha)= \sum_{i=0}^{s-1} f_i G_i^{(d)}(t),$$
where $f_i$ are real number and $f_i=0$ for $i\equiv s \mod 2$. Then, 
\begin{eqnarray}
|X| \leq 2 \sum_{i \text{ with } f_i>0} h_i, 
\end{eqnarray} 
where the summation runs through $0\leq i \leq s$ satisfying $f_i>0$.  
\end{theorem}

If $f_i>0$ for all $0\leq i\leq s$ (antipodal case: $f_i>0$ for all $i\equiv s-1 \mod 2$), then this upper bound is the same as Fisher type inequality. Therefore, this upper bound is better than the Fisher type inequality. 
\section{Proof of main theorems}
First, we prepare the two results to prove the main theorems. 
Let $\{ \varphi_{l,k}\}_{1\leq k \leq h_l}$ be an orthonormal basis of ${\rm Harm}_l(\mathbb{R}^d)$ with respect to $\langle \ast, \ast \rangle$, where $\langle f, g \rangle:= \frac{1}{|S^{d-1}|}\int_{S^{d-1}} f(x)g(x) d\sigma(x)$. 
\begin{theorem}[Addition formula \cite{DGS,Bannai-Bannai-book}] 
For any $x$, $y$ on $S^{d-1}$, we have
$$
\sum_{k=1}^{h_l} \varphi_{l,k}(x) \varphi_{l,k}(y)= G_l^{(d)}((x,y)).
$$
\end{theorem} 

The following lemma is elementary. 
\begin{lemma}\label{key lemma}
Let $M$ be a symmetric matrix in $M_n(\mathbb{R})$ and $N$ be an $m \times n$ matrix. $^tN$ means the transpose matrix of $N$. $D_{u,v}$ is an $m \times m$ diagonal matrix such that the number of positive entries  is $u$ and the number of negative entries is $v$. If the equality 
$$
N M {^tN}=D_{u,v}
$$
holds, then the number of the positive (resp.\ negative) eigenvalues of $M$ is bounded below by $u$ (resp. $v$). 
\end{lemma}

\begin{proof}[Proof of Theorem \ref{main}]
Let $X:=\{x_1,x_2,\ldots,x_{|X|}\}$ be an $s$-distance set on $S^{d-1}$. 
Let $\{ \varphi_{l,k}\}_{1\leq k \leq h_l}$ be an orthonormal basis of ${\rm Harm}_l(\mathbb{R}^d)$.
$H_l$ is the characteristic matrix that is indexed by $X$ and an orthonormal basis $\{ \varphi_{l,k}\}_{1\leq k \leq h_l}$ and  whose entry is defined by $H_{l}(x_i,\varphi_{l,j}):=\varphi_{l,j}(x_i)$. Then,  
$$
[f_0 H_0, f_1 H_1, \ldots, f_s H_s] \quad \text{ and } \quad [ H_0,  H_1, \ldots, H_s]
$$
are $|X| \times \sum_{i=0}^s h_i$ matrices. 
We have the Gegenbauer expansion $F_X(t)=\prod_{\alpha \in A(X)}\frac{t-\alpha}{1-\alpha}= \sum_{i=0}^s f_i G_i^{(d)}(t)$. By the addition formula, 
$$I_{|X|}=[f_0 H_0, f_1 H_1, \ldots, f_s H_s] 
\left[\begin{array}{c}
^t{H_0}\\
^t{H_1}\\
\vdots\\
^t{H_s}
\end{array}
\right] 
=
[H_0, H_1 , \ldots , H_s] {\rm Diag}\left[
\begin{array}{c}
f_0\\
 f_1\\
 \vdots\\
 f_1\\
 \vdots\\
 f_s\\
 \vdots\\ 
 f_s   
\end{array}
\right]
\left[\begin{array}{c}
^t{H_0}\\
^t{H_1}\\
\vdots\\
^t{H_s}
\end{array}
\right] 
,
$$
where $I_{|X|}$ is the identity matrix of degree $|X|$, ${\rm Diag}[\ast]$ means a diagonal matrix and the number of entries $f_i$ is $h_i$. By Lemma \ref{key lemma}, 
$$
|X| \leq \sum_{i \text{ with } f_i>0 }h_i.
$$
\end{proof}
We can prove the antipodal case by using the same method as above proof.

\section{Examples}
In this section, we introduce some examples which attain the upper bound in the main theorems. 
\subsection{The case $s=1$, $f_0>0$ and $f_1 \leq 0$}
\begin{corollary} 
Let $X$ be a $1$-distance set and $A(X)=\{ \alpha \}$. Then, 
$F_X(t):=t-\alpha =\sum_{i=0}^1 f_i G_i^{(d)}(t)$ where $f_0= 1/d$ and $f_1=-\alpha$. 
If $\alpha \geq 0$, then 
$$
|X| \leq  h_1= d.
$$  
\end{corollary} 
Clearly, a $(d-1)$-dimensional regular simplex with a nonnegative inner product on $S^{d-1}$ attains this upper bound.  

\subsection{The case $s=2$, $f_0>0,f_1\leq 0$ and $f_2>0$}
\begin{corollary} \label{musin}
Let $X$ be a $2$-distance set and $A(X)=\{ \alpha, \beta \}$. Then, 
$F_X(t):=(t-\alpha )(t-\beta)=\sum_{i=0}^2 f_i G_i^{(d)}(t)$ where $f_0=\alpha \beta + 1/d$, $f_1=-(\alpha+\beta)/d$ and $f_2=2/(d(d+2))$. 
If $\alpha +\beta \geq 0$, then 
$$
|X| \leq h_0 + h_2= \binom{d+1}{2}.
$$  
\end{corollary} 
Musin proved this corollary by using a polynomial method in \cite{Musin}. The following examples attain this upper bound. 
\begin{example}
Let $U_d$ be a $d$-dimensional regular simplex. We define
$$X:=\left\{ \left. \frac{x+y}{2} \right| x,y \in U_d, x \ne y \right\}$$  for $d \geq 7$.
Then, $X$ is a $2$-distance set on $S^{d-1}$, $|X|=\binom{d+1}{2}$, $f_0>0$, $f_1\leq 0$ and $f_2>0$.  
\end{example}

\subsection{Examples from tight spherical $(2s-1)$-designs}
\begin{corollary} 
Let $X$ be an $s$-distance set on $S^{d-1}$. We have the Gegenbauer expansion $F_X(t)=\prod_{\alpha \in A(X)}(t-\alpha)=\sum_{i=0}^s f_i G_i^{(d)}(t)$. If $f_i > 0$ for all $i \equiv s \mod 2$ and $f_i \leq 0$ for all $i \equiv s-1 \mod 2$, then 
$$
|X| \leq \sum_{i=0}^{\lfloor \frac{s}{2} \rfloor}h_{s-2i}= \binom{d+s-2}{s-1}.
$$
\end{corollary} 
The following examples attain above upper bound. 
\begin{example}
Let $X$ be a tight spherical $(2s-1)$-design, that is, $X$ is an antipodal $s$-distance set which attains the Fisher type inequality \cite{DGS}. There exist a subset $Y$ such that $X= Y \cup (-Y)$ and $|X|=2|Y|$. $Y$ is an $(s-1)$-distance set and $F_Y(t):=\sum_{i=0}^{s-1}f_iG_i^{(d)}(t)$. Then, $f_i = 0$ for all $i \equiv s-2 \mod 2$ and $f_i > 0$ for all $i \equiv s-1 \mod 2$ and $|Y|=\binom{d+s-3}{s-2}$.   
\end{example}

\end{document}